\documentclass[oneside,english]{amsart}
\usepackage[T1]{fontenc}
\usepackage[latin9]{inputenc}
\usepackage{prettyref}
\usepackage{amsthm}
\usepackage{amssymb}

\makeatletter
\theoremstyle{plain}
\newtheorem{thm}{\protect\theoremname}
\theoremstyle{plain}
\newtheorem{cor}[thm]{\protect\corollaryname}
\theoremstyle{definition}
\newtheorem{example}[thm]{\protect\examplename}

\usepackage{braket}

\newrefformat{cor}{Corollary \ref{#1}}

\makeatother

\usepackage{babel}
\providecommand{\corollaryname}{Corollary}
\providecommand{\examplename}{Example}
\providecommand{\theoremname}{Theorem}

\begin{document}

\title{Game arguments in some existence theorems of Friedberg numberings}

\author{Takuma Imamura}

\address{Department of Mathematics\\
University of Toyama\\
3190 Gofuku, Toyama 930-8555, Japan}

\curraddr{Research Institute for Mathematical Sciences\\
Kyoto University\\
Kitashirakawa-Oiwake-cho, Sakyo-ku, Kyoto 606-8502, Japan}

\email{timamura@kurims.kyoto-u.ac.jp}
\begin{abstract}
We provide game-theoretic proofs of some well-known existence theorems
of Friedberg numberings for the class of all partial computable functions,
including (1) the existence of two incomparable Friedberg numberings;
(2) the existence of a uniformly c.e. sequence of pairwise incomparable
Friedberg numberings; (3) the existence of a uniformly c.e. independent
sequence of Friedberg numberings. Parameterizing these proofs, we
have game-theoretic proofs of Kummer's criteria and their modifications.
\end{abstract}

\maketitle

\section{Introduction}

Rogers \cite{Rog58} introduced the notion of computable numbering
of all partial computable functions and the notion of reducibility.
He showed that the set of all equivalence classes of computable numberings
forms an upper semilattice with respect to reducibility, nowadays
called the Rogers semilattice. He asked whether this semilattice is
a lattice, and if not, whether any two elements have a lower bound.
Friedberg \cite{Fri58} constructed an injective computable numbering
(called a Friedberg numbering) by a finite-injury priority argument.
Every Friedberg numbering is minimal in the Rogers semilattice. Pour-El
\cite{Pou64} showed that there are two incomparable Friedberg numberings
through modifying Friedberg's construction. Thus both of Rogers' questions
were negatively answered.

Shen \cite{She12} gave some examples of game-theoretic arguments
in computability theory and algorithmic information theory. In particular,
he gave a game-theoretic proof of Friedberg's theorem. The game representation
of Friedberg's construction is clean, and can be used to prove other
existence theorems of Friedberg numberings as we will demonstrate
in this paper.

In \prettyref{sec:Friedberg's construction}, we present Shen's proof
for later use. We provide game-theoretic proofs of two well-known
criteria for c.e. classes of partial computable functions to have
a Friedberg numbering. In \prettyref{sec:Modifications}, we give
two game-theoretic proofs of Pour-El's result using two different
games. We also give game-theoretic proofs of some existence criteria
of an infinite c.e. sequence and an independent sequence of Friedberg
numberings. These proofs are essentially modifications of Shen's proof.

\section{Notations and Definitions}

We denote by $\mathcal{P}^{\left(1\right)}$ the set of all partial
computable functions from $\mathbb{N}$ to $\mathbb{N}$. $\braket{\cdot,\cdot}$
is a computable pairing function which is a computable bijection between
$\mathbb{N}^{2}$ and $\mathbb{N}$.

Let $\mathcal{A}$ be any set. A surjective map $\nu\colon\mathbb{N}\to\mathcal{A}$
is called a \emph{numbering} of $\mathcal{A}$. Let $\nu$ and $\mu$
be numberings of $\mathcal{A}$. We say that $\nu$ is \emph{reducible
to} $\mu$, denoted by $\nu\leq\mu$, if there exists a total computable
function $f\colon\mathbb{N}\to\mathbb{N}$ such that $\nu=\mu\circ f$.
We say that $\nu$ and $\mu$ are \emph{equivalent} if they are reducible
to each other. We say that $\nu$ and $\mu$ are \emph{incomparable}
if they are not reducible to each other.

We often identify a numbering $\nu$ of a set of partial maps from
$X$ to $Y$ with the partial map $\nu\left(i,x\right)=\nu\left(i\right)\left(x\right)$
from $\mathbb{N}\times X$ to $Y$. A numbering $\nu$ of a subset
of $\mathcal{P}^{\left(1\right)}$ is said to be \emph{computable}
if it is computable as a partial function from $\mathbb{N}^{2}$ to
$\mathbb{N}$. A computable injective numbering is called a \emph{Friedberg
numbering}. A sequence $\set{\nu_{i}}_{i\in\mathbb{N}}$ of numberings
of a subset of $\mathcal{P}^{\left(1\right)}$ is said to be \emph{uniformly
c.e.} if it is uniformly c.e. as a sequence of partial functions from
$\mathbb{N}^{2}$ to $\mathbb{N}$, or equivalently, if it is computable
as a partial function from $\mathbb{N}^{3}$ to $\mathbb{N}$. We
say that a sequence $\set{\nu_{i}}_{i\in\mathbb{N}}$ of numberings
of a set $\mathcal{A}$ is \emph{independent} if $\nu_{i}\nleq\bigoplus_{j\neq i}\nu_{j}$
for all $i\in\mathbb{N}$, where $\bigoplus_{i\in\mathbb{N}}\nu_{j}$
is the direct sum of $\set{\nu_{i}}_{i\in\mathbb{N}}$ defined by
$\bigoplus_{i\in\mathbb{N}}\nu_{i}\left(\braket{j,k}\right)=\nu_{j}\left(k\right)$.

Let $\mathcal{A}$ and $\mathcal{B}$ be subsets of $\mathcal{P}^{\left(1\right)}$.
We say that $\mathcal{B}$ is \emph{strongly dense in} $\mathcal{A}$
provided that every finite subfunction of a member of $\mathcal{A}$
has infinitely many extensions in $\mathcal{B}$.

\section{\label{sec:Friedberg's construction}Friedberg's construction: the
infinite game with two boards}
\begin{thm}[{Friedberg \cite[Corollary to Theorem 3]{Fri58}}]
\label{thm:Fri58-Corollary-to-Theorem3}$\mathcal{P}^{\left(1\right)}$
has a Friedberg numbering.
\end{thm}

\begin{proof}[Proof (Shen \cite{She12})]
First, we consider an infinite game $\mathcal{G}_{0}$ and prove
that the existence of a computable winning strategy of $\mathcal{G}_{0}$
for the second player implies the existence of a Friedberg numbering
of $\mathcal{P}^{\left(1\right)}$. The game $\mathcal{G}_{0}$ is
defined as follows:

\begin{description}
\item [{Players}] Alice, Bob.
\item [{Protocol}] FOR $s=0,1,2,\ldots$:

Alice announces a finite partial function $A_{s}\colon\mathbb{N}^{2}\rightharpoonup\mathbb{N}$.

Bob announces a finite partial function $B_{s}\colon\mathbb{N}^{2}\rightharpoonup\mathbb{N}$.
\item [{Collateral duties}] $A_{s}\subseteq A_{s+1}$ and $B_{s}\subseteq B_{s+1}$
for all $s\in\mathbb{N}$.
\item [{Winner}] Let $A=\bigcup_{s\in\mathbb{N}}A_{s}$ and $B=\bigcup_{s\in\mathbb{N}}B_{s}$.
Bob wins if

\begin{enumerate}
\item for each $i\in\mathbb{N}$, there exists a $j\in\mathbb{N}$ such
that $A\left(i,\cdot\right)=B\left(j,\cdot\right)$;
\item for any $i,j\in\mathbb{N}$, if $i\neq j$, then $B\left(i,\cdot\right)\neq B\left(j,\cdot\right)$.
\end{enumerate}
\end{description}
We consider $A$ and $B$ as two boards, $A$-table and $B$-table.
Each board is a table with an infinite number of rows and columns.
Each player plays on its board. At each move player can fill finitely
many cells with any natural numbers. The collateral duties prohibit
players from erasing cells.

A strategy is a map that determines the next action based on the previous
actions of the opponent. Since any action in this game is a finitary
object, we can define the computability of strategies via g\"odelization.
Suppose that there exists a computable winning strategy for Bob. Let
Alice fill $A$-table with the values of some computable numbering
of $\mathcal{P}^{\left(1\right)}$ by using its finite approximation,
and let Bob use some computable winning strategy. Clearly $B$ is
a Friedberg numbering of $\mathcal{P}^{\left(1\right)}$.

Second, we consider an infinite game $\mathcal{G}_{1}$, which is
a modification of $\mathcal{G}_{0}$, and describe a computable winning
strategy of $\mathcal{G}_{1}$. The game $\mathcal{G}_{1}$ is as
follows:
\begin{description}
\item [{Players}] Alice, Bob.
\item [{Protocol}] FOR $s=0,1,2,\ldots$:

Alice announces a finite partial function $A_{s}\colon\mathbb{N}^{2}\rightharpoonup\mathbb{N}$.

Bob announces a finite partial function $B_{s}\colon\mathbb{N}^{2}\rightharpoonup\mathbb{N}$
and a finite set $K_{s}\subseteq\mathbb{N}$.
\item [{Collateral duties}] $A_{s}\subseteq A_{s+1}$, $B_{s}\subseteq B_{s+1}$
and $K_{s}\subseteq K_{s+1}$ for all $s\in\mathbb{N}$.
\item [{Winner}] Let $A=\bigcup_{s\in\mathbb{N}}A_{s}$, $B=\bigcup_{s\in\mathbb{N}}B_{s}$
and $K=\bigcup_{s\in\mathbb{N}}K_{s}$. Bob wins if

\begin{enumerate}
\item for each $i\in\mathbb{N}$, there exists a $j\in\mathbb{N}\setminus K$
such that $A\left(i,\cdot\right)=B\left(j,\cdot\right)$;
\item for any $i,j\in\mathbb{N}\setminus K$, if $i\neq j$, then $B\left(i,\cdot\right)\neq B\left(j,\cdot\right)$.
\end{enumerate}
\end{description}
We consider that in this game Bob can \emph{invalidate} some rows
and that we ignore invalid rows when we decide the winner. Bob cannot
validate invalid rows again.

To win this game, Bob hires a countable number of assistants who guarantee
that each of the rows in $A$-table appears in $B$-table exactly
once. They can reserve rows in $B$-table exclusively, fill their
reserved rows, and invalidate their reserved rows. At each move, the
assistants work one by one. To ensure that only finitely many assistants
work at each move, the $i$-th assistant starts working at move $i$.
The instruction for the $i$-th assistant is as follows: \textit{if
you have no reserved row, reserve the first unused row. Let $k$ be
the number of rows such that you have already invalidated. If, in
the current state of $A$-table, the first $k$ positions of the $i$-th
row are identical to the first $k$ positions of some previous row,
invalidate your reserved row. If you have a reserved row, copy the
current contents of the $i$-th row of $A$-table into your reserved
row.} These instructions guarantee, in the limit, that
\begin{enumerate}
\item each of the rows in $B$-table has been reserved or invalidated;
\item if the $i$-th row in $A$-table is identical to some previous row,
then the $i$-th assistant invalidates reserved rows infinitely many
times, so there is no permanently reserved row;
\item if not, the $i$-th assistant invalidates reserved rows only finitely
many times, so there is a permanently reserved row.
\end{enumerate}
\noindent In the second case, the $i$-th assistant faithfully copies
the contents of the $i$-th row of $A$-table into the permanently
reserved row. As a consequence, Bob wins the simplified game. Now
we verify the above three claims. The first claim is trivial, because
each assistant reserves the first unused row at least once. To verify
the second and the third one, suppose that the $i$-th row in $A$-table
is not identical to any previous row in the limit. For each previous
row, select some column witnessing that this row is not identical
to the $i$-th row. Let $k$ be the maximum of the selected columns.
Wait for convergence of the rectangular area $\left[0,i\right]\times\left[0,k\right]$
of $A$-table. After that, the first $k$ positions of the $i$-th
row in $A$-table are not identical to the first $k$ positions of
any previous row, and hence the $i$-th assistant invalidates reserved
rows at most $k$ times. Conversely, suppose that the $i$-th assistant
invalidates reserved rows only finitely many times. Let $k$ be the
number of invalidations. After the $k$-th invalidation, the $i$-th
row in $A$-table is not identical to any previous row, and the same
is true in the limit.

Finally, we describe a computable winning strategy of $\mathcal{G}_{0}$
through modifying the winning strategy of $\mathcal{G}_{1}$ described
above. We say that a row is odd if it contains a finite odd number
of non-empty cells. We can assume without loss of generality that
odd rows never appear in $A$-table: if Alice fills some cells in
a row making this row odd, Bob ignores one of these cells until Alice
fills other cells in this row. We replace invalidation with \emph{odd-ification}:
instead of invalidating a row, fill some cells in this row making
it new and odd. Bob consider that odd-ified rows in $B$-table are
invalid. This modification guarantees that each of the non-odd rows
of $A$-table appears in $B$-table exactly once. Bob hires an additional
assistant who guarantees that each odd row appears in $B$-table exactly
once. At each move, the additional assistant reserves some row exclusively
and fills some cells in this row making it new and odd so that all
odd rows are exhausted in the limit. Thus Bob wins this game, and
the theorem is proved.
\end{proof}
Kummer \cite{Kum90a} gave a priority-free proof of the existence
of a Friedberg numbering of $\mathcal{P}^{\left(1\right)}$. Let us
split $\mathcal{P}^{\left(1\right)}$ into $\mathcal{P}^{\left(1\right)}\setminus\mathcal{O}$
and $\mathcal{O}$, where $\mathcal{O}$ is the set of all \emph{odd}
partial functions. The key observation is that $\mathcal{P}^{\left(1\right)}\setminus\mathcal{O}$
has a computable numbering, $\mathcal{O}$ has a Friedberg numbering,
and $\mathcal{O}$ is strongly dense in $\mathcal{P}^{\left(1\right)}\setminus\mathcal{O}$.
He provided the following useful criterion.
\begin{cor}[{Kummer \cite[Extension Lemma]{Kum89}}]
\label{cor:Kum89-Extension-Lemma}Let $\mathcal{A}$ and $\mathcal{B}$
be disjoint subsets of $\mathcal{P}^{\left(1\right)}$. If $\mathcal{A}$
has a computable numbering, $\mathcal{B}$ has a Friedberg numbering,
and $\mathcal{B}$ is strongly dense in $\mathcal{A}$, then $\mathcal{A}\cup\mathcal{B}$
has a Friedberg numbering.
\end{cor}

\begin{proof}
Let us play the game $\mathcal{G}_{0}$ where Alice fills $A$-table
with the values of some computable numbering of $\mathcal{A}$ and
Bob uses the above-mentioned strategy of $\mathcal{G}_{0}$ which
will be modified as follows. In this strategy, Bob does not ignore
cells in $A$-table, so that we do not assume that odd rows never
appear in $A$-table. The assistants use partial functions in $\mathcal{B}$
instead of odd partial functions. Replace odd-ification with \emph{$\mathcal{B}$-ification}:
instead of odd-ificating a row, fill some cells in this row making
it an unused member of $\mathcal{B}$ in the limit. The additional
assistant guarantees that each member of $\mathcal{B}$ appears in
$B$-table exactly once. Then, Bob wins, and $B$ becomes a Friedberg
numbering of $\mathcal{A}\cup\mathcal{B}$. Note that it is (effectively)
possible to find an unused member of $\mathcal{B}$ that extends a
given partial function $f$: let $\nu$ be a Friedberg numbering of
$\mathcal{B}$ and let $\set{\nu_{s}}_{s\in\mathbb{N}}$ be a finite
approximation of $\nu$. Recall that $f$ is a finite subfunction
of a member of $\mathcal{A}$, and has infinitely many extensions
in $\mathcal{B}$. At each stage, only finitely many members of $\mathcal{B}$
have been used. Hence one can find an $s\in\mathbb{N}$ and an unused
index $i\in\mathbb{N}$ such that $\nu_{s}\left(i\right)$ extends
$f$.
\end{proof}
\begin{cor}[{Pour-El and Putnam \cite[Theorem 1]{PP65}}]
Let $\mathcal{A}$ be a subset of $\mathcal{P}^{\left(1\right)}$
and $f$ be a member of $\mathcal{P}^{\left(1\right)}$ with an infinite
domain. If $\mathcal{A}$ has a computable numbering, then there exists
a subset $\mathcal{B}$ of $\mathcal{P}^{\left(1\right)}$ such that

\begin{enumerate}
\item $\mathcal{A}\subseteq\mathcal{B}$;
\item the domain of every member of $\mathcal{B}\setminus\mathcal{A}$ is
finite;
\item for any $g\in\mathcal{B}\setminus\mathcal{A}$, there exists an $h\in\mathcal{A}$
with $g\subseteq f\cup h$;
\item $\mathcal{B}$ has a Friedberg numbering.
\end{enumerate}
\end{cor}

\begin{proof}
Let us play the game $\mathcal{G}_{0}$ where Alice fills $A$-table
with the values of some computable numbering of $\mathcal{P}^{\left(1\right)}$
and Bob uses the winning strategy of $\mathcal{G}_{0}$ modified as
follows. When an assistant fills a cell in the $j$-th column making
this row odd, the value $f\left(j\right)$ must be used for filling.
Modify the instruction for the additional assistant as follows: \textit{if
there is an odd row $i$ in $A$-table such that this row is not identical
to any row in $B$-table, reserve the first unused row, copy the current
contents of the $i$-th row of $A$-table into your reserved row,
and release your reserved row. }Released rows cannot be used forever.
Here, exceptionally, the additional assistant does not ignore cells
in $A$-table. Then, Bob wins, $\mathcal{B}=\set{B\left(i,\cdot\right)|i\in\mathbb{N}}$
has the desired properties, and $B$ becomes a Friedberg numbering
of $\mathcal{B}$.
\end{proof}

\section{\label{sec:Modifications}Modifications: more boards}
\begin{thm}[{Pour-El \cite[Theorem 2]{Pou64}}]
\label{thm:Pou64-Theorem2}There are two incomparable Friedberg numberings
of $\mathcal{P}^{\left(1\right)}$.
\end{thm}

The first proof is obtained from the proof of \prettyref{thm:Fri58-Corollary-to-Theorem3}
through modifying in the same way done by Pour-El.
\begin{proof}[Proof (asymmetric version)]
We consider the following game $\mathcal{G}_{2}$:

\begin{description}
\item [{Players}] Alice, Bob.
\item [{Protocol}] FOR $s=0,1,2,\ldots$:

Alice announces a finite partial function $A_{s}\colon\mathbb{N}^{2}\rightharpoonup\mathbb{N}$.

Bob announces a finite partial function $B_{s}\colon\mathbb{N}^{2}\rightharpoonup\mathbb{N}$.
\item [{Collateral duties}] $A_{s}\subseteq A_{s+1}$ and $B_{s}\subseteq B_{s+1}$
for all $s\in\mathbb{N}$.
\item [{Winner}] Let $A=\bigcup_{s\in\mathbb{N}}A_{s}$ and $B=\bigcup_{s\in\mathbb{N}}B_{s}$.
Bob wins if

\begin{enumerate}
\item for each $i\in\mathbb{N}$, there exists a $j\in\mathbb{N}$ such
that $A\left(i,\cdot\right)=B\left(j,\cdot\right)$;
\item for any $i,j\in\mathbb{N}$, if $i\neq j$, then $B\left(i,\cdot\right)\neq B\left(j,\cdot\right)$;
\item for each $i\in\mathbb{N}$, if $A\left(i,\cdot\right)$ is total,
then there exists a $j\in\mathbb{N}$ such that $B\left(A\left(i,j\right),\cdot\right)\neq A\left(j,\cdot\right)$.
\end{enumerate}
\end{description}
Suppose that there exists a computable winning strategy of $\mathcal{G}_{2}$
for Bob. Let Alice fill $A$-table with the values of some Friedberg
numbering of $\mathcal{P}^{\left(1\right)}$, and let Bob use some
computable winning strategy. Then, $A$ and $B$ are Friedberg numberings
of $\mathcal{P}^{\left(1\right)}$, and $A$ is not reducible to $B$.
Since $A$ is minimal, $B$ is also not reducible to $A$.

We describe a computable winning strategy of $\mathcal{G}_{2}$. Bob
uses the winning strategy of $\mathcal{G}_{0}$ described in the proof
of \prettyref{thm:Fri58-Corollary-to-Theorem3}, which guarantees
that the first two winning conditions are satisfied. For the third
winning condition, Bob promulgates the following instruction for the
$i$-th assistant: \textit{if the $\left(i,i\right)$-element of $A$-table
has been filled, and you have reserved the $A\left(i,i\right)$-th
row, then odd-ify the $A\left(i,i\right)$-th row.} Each of these
instructions is done at most once, because, after doing that, the
corresponding requirement is permanently satisfied. Hence they do
not interrupt satisfying the first two winning conditions. It remains
to show that the third winning condition is also satisfied. Suppose
that $A\left(i,\cdot\right)$ is total. We can assume without loss
of generality that $i$ is the least index of $A\left(i,\cdot\right)$,
i.e., there is no $j<i$ with $A\left(j,\cdot\right)=A\left(i,\cdot\right)$.
Since $A\left(i,\cdot\right)$ is not either odd or even, the $i$-th
assistant has a permanently reserved row $j$. The additional instruction
guarantees that $j\neq A\left(i,i\right)$. By the second winning
condition, we have that $B\left(A\left(i,i\right),\cdot\right)\neq B\left(j,\cdot\right)=A\left(i,\cdot\right)$.
Thus Bob wins this game.
\end{proof}
The second proof is more symmetric and parameterized. In the second
proof, we consider an infinite game where Alice constructs two partial
functions $A$ and $R$, and Bob constructs two partial functions
$B$ and $C$. Bob wins if each of $B$ and $C$ contains all the
partial functions from $A$ exactly once; and $B$ and $C$ are not
reducible to each other by any total function from $R$. Otherwise,
Alice wins. An outline of Bob's winning strategy is as follows: let
Bob fix a computable injective numbering $\varphi$ of non-odd partial
functions, which will be used as a witness of incomparability of $B$
and $C$ relative to $R$. The assistant responsible for $\varphi_{2i}$
guarantees that $R\left(i,\cdot\right)$ is not a reduction function
from $C$ to $B$. The assistant responsible for $\varphi_{2i+1}$
guarantees that $R\left(i,\cdot\right)$ is not a reduction function
from $B$ to $C$. We now describe the details.
\begin{proof}[Proof (symmetric version)]
Consider the following game $\mathcal{G}_{3}$:

\begin{description}
\item [{Players}] Alice, Bob.
\item [{Protocol}] FOR $s=0,1,2,\ldots$:

Alice announces two finite partial functions $A_{s},R_{s}\colon\mathbb{N}^{2}\rightharpoonup\mathbb{N}$.

Bob announces two finite partial functions $B_{s},C_{s}\colon\mathbb{N}^{2}\rightharpoonup\mathbb{N}$.
\item [{Collateral duties}] $A_{s}\subseteq A_{s+1}$, $R_{s}\subseteq R_{s+1}$,
$B_{s}\subseteq B_{s+1}$ and $C_{s}\subseteq C_{s+1}$ for all $s\in\mathbb{N}$.
\item [{Winner}] Let $A=\bigcup_{s\in\mathbb{N}}A_{s}$, $R=\bigcup_{s\in\mathbb{N}}R_{s}$,
$B=\bigcup_{s\in\mathbb{N}}B_{s}$ and $C=\bigcup_{s\in\mathbb{N}}C_{s}$.
Bob wins if

\begin{enumerate}
\item for each $i\in\mathbb{N}$, there exists a $j\in\mathbb{N}$ such
that $A\left(i,\cdot\right)=B\left(j,\cdot\right)$;
\item for each $i\in\mathbb{N}$, there exists a $j\in\mathbb{N}$ such
that $A\left(i,\cdot\right)=C\left(j,\cdot\right)$;
\item for any $i,j\in\mathbb{N}$, if $i\neq j$, then $B\left(i,\cdot\right)\neq B\left(j,\cdot\right)$;
\item for any $i,j\in\mathbb{N}$, if $i\neq j$, then $C\left(i,\cdot\right)\neq C\left(j,\cdot\right)$;
\item for each $i\in\mathbb{N}$, if $R\left(i,\cdot\right)$ is total,
then there exists a $j\in\mathbb{N}$ such that $B\left(R\left(i,j\right),\cdot\right)\neq C\left(j,\cdot\right)$;
\item for each $i\in\mathbb{N}$, if $R\left(i,\cdot\right)$ is total,
then there exists a $j\in\mathbb{N}$ such that $C\left(R\left(i,j\right),\cdot\right)\neq B\left(j,\cdot\right)$.
\end{enumerate}
\end{description}
Firstly, to consider the constant functions separately, we modify
the winning strategy of $\mathcal{G}_{0}$ described in the proof
of \prettyref{thm:Fri58-Corollary-to-Theorem3}. Bob promulgates the
following instruction for the $i$-th assistant: \textit{let $k$
be the number of rows such that you have already odd-ified. If, in
the current state of $A$-table, the first $k$ positions of the $i$-th
row are constant, odd-ify your reserved row.} Next, Bob hires a countable
number of additional assistants who guarantee that each of the constant
functions appears in $B$-table exactly once. At each move, all assistants
work one by one. The $i$-th additional assistant starts working at
move $i$. The instruction for the $i$-th additional assistant is
as follows: \textit{if you have no reserved row, reserve the first
unused row. If you have a reserved row, fill your reserved row so
that this row becomes the constant function $i$ in the limit.} The
modified strategy is still a computable winning strategy of $\mathcal{G}_{0}$.

Secondly, we describe a computable winning strategy of $\mathcal{G}_{3}$.
To guarantee that the first four winning conditions are satisfied,
Bob uses, for $B$-table and $C$-table, two copies of the winning
strategy of $\mathcal{G}_{0}$ described in the previous paragraph.
At each move, the strategies work one by one. For the fifth winning
condition, Bob promulgates the following instruction for the $2i$-th
additional assistant: \textit{if you have reserved the $j$-th row
in $C$-table, the $\left(i,j\right)$-element of $R$-table has been
filled, and you have reserved the $R\left(i,j\right)$-th row in $B$-table,
then odd-ify the $R\left(i,j\right)$-th row in $B$-table.} Symmetrically,
for the sixth condition, Bob promulgates the following instruction
for the $\left(2i+1\right)$-st additional assistant: \textit{if you
have reserved the $j$-th row in $B$-table, the $\left(i,j\right)$-element
of $R$-table has been filled, and you have reserved the $R\left(i,j\right)$-th
row in $C$-table, then odd-ify the $R\left(i,j\right)$-th row in
$C$-table.} Since each of these instructions is done at most once,
the first four winning conditions are still satisfied. We now prove
that the last two winning conditions are also satisfied. Suppose that
$R\left(i,\cdot\right)$ is total. The $2i$-th additional assistant
has a permanently reserved row $c$ in $C$-table. The additional
instruction guarantees that the $2i$-th assistant does not permanently
reserve the $R\left(i,c\right)$-th row in $B$-table. Hence $B\left(R\left(i,c\right),\cdot\right)$
is not identical to the constant function $2i$. It follows that $B\left(R\left(i,c\right),\cdot\right)\neq C\left(c,\cdot\right)$.
Similarly, we have that $C\left(R\left(i,b\right),\cdot\right)\neq B\left(b,\cdot\right)$
for some $b$. Thus Bob wins this game.
\end{proof}
The following can be proved in a way similar to the second proof of
\prettyref{thm:Pou64-Theorem2}. It suffices to consider an infinite
game where Bob constructs a countable number of partial functions
$B^{0},B^{1},B^{2},\ldots$ simultaneously. To ensure finiteness of
actions, Bob shall start constructing the $i$-th numbering $B^{i}$
at move $i$.
\begin{cor}[{Khutoretskii \cite[Corollary 2]{Khu69}}]
\label{cor:Khu69-Corollary2}There exists a uniformly c.e. sequence
of pairwise incomparable Friedberg numberings of $\mathcal{P}^{\left(1\right)}$.
\end{cor}

We can also prove the following criterion in the same way.
\begin{cor}
\label{cor:another-Extension-Lemma}Let $\mathcal{A}$, $\mathcal{B}$
and $\mathcal{C}$ be disjoint subsets of $\mathcal{P}^{\left(1\right)}$.
If $\mathcal{A}$ has a computable numbering, $\mathcal{B}$ and $\mathcal{C}$
have Friedberg numberings, and $\mathcal{B}$ is strongly dense in
$\mathcal{A}\cup\mathcal{C}$, then there exists a uniformly c.e.
sequence of pairwise incomparable Friedberg numberings of $\mathcal{A}\cup\mathcal{B}\cup\mathcal{C}$.
\end{cor}

\prettyref{cor:Khu69-Corollary2} is an instance of \prettyref{cor:another-Extension-Lemma}
where $\mathcal{B}$ is the set of all odd partial functions, $\mathcal{C}$
is the set of all constant functions, and $\mathcal{A}$ is the complement
of $\mathcal{B}\cup\mathcal{C}$.
\begin{cor}[{Kummer \cite[Lemma]{Kum90b}}]
\label{cor:Kum90b-Lemma}Let $\mathcal{A}$ and $\mathcal{B}$ be
disjoint subsets of $\mathcal{P}^{\left(1\right)}$. If $\mathcal{A}$
and $\mathcal{B}$ have Friedberg numberings, and $\mathcal{B}$ is
strongly dense in $\mathcal{A}$, then there exists a uniformly c.e.
sequence of pairwise incomparable Friedberg numberings of $\mathcal{A}\cup\mathcal{B}$.
\end{cor}

\begin{proof}
Fix a Friedberg numbering $\nu$ of $\mathcal{A}$. Apply \prettyref{cor:another-Extension-Lemma}
to $\mathcal{A}'=\set{\nu\left(2i\right)|i\in\mathbb{N}}$, $\mathcal{B}$
and $\mathcal{C}=\set{\nu\left(2i+1\right)|i\in\mathbb{N}}$.
\end{proof}
Conversely, we can prove \prettyref{cor:another-Extension-Lemma}
using \prettyref{cor:Kum90b-Lemma}. Applying \prettyref{cor:Kum89-Extension-Lemma}
to $\mathcal{A}$ and $\mathcal{C}$, we have a Friedberg numbering
of $\mathcal{A}\cup\mathcal{C}$. By \prettyref{cor:Kum90b-Lemma},
$\left(\mathcal{A}\cup\mathcal{C}\right)\cup\mathcal{B}$ has a uniformly
c.e. sequence of pairwise incomparable Friedberg numberings.
\begin{cor}[{Kummer \cite[Theorem 1]{Kum90b}}]
\label{cor:Kum90b-Theorem1}Let $\mathcal{A}$ be a subset of $\mathcal{P}^{\left(1\right)}$.
If $\mathcal{A}$ has a Friedberg numbering, and $\mathcal{A}$ is
strongly dense in itself, then there exists a uniformly c.e. sequence
of pairwise incomparable Friedberg numberings of $\mathcal{A}$.
\end{cor}

\begin{proof}
See \cite{Kum90b}.
\end{proof}
It is immediate from \prettyref{cor:Khu69-Corollary2} that the Rogers
semilattice of $\mathcal{P}^{\left(1\right)}$ has a countable antichain.
Naturally, we can ask whether the Rogers semilattice of $\mathcal{P}^{\left(1\right)}$
has a countable chain. The answer is positive. This was proved by
Khutoretskii \cite[Proof of Corollary 1]{Khu71}. Moreover, there
exists a uniformly c.e. independent sequence of Friedberg numberings
of $\mathcal{P}^{\left(1\right)}$. Consequently, the Rogers semilattice
of $\mathcal{P}^{\left(1\right)}$ is a universal countable partial
order.
\begin{thm}
There exists a uniformly c.e. independent sequence of Friedberg numberings
of $\mathcal{P}^{\left(1\right)}$.
\end{thm}

\begin{proof}
Consider the following game $\mathcal{G}_{4}$:

\begin{description}
\item [{Players}] Alice, Bob.
\item [{Protocol}] FOR $s=0,1,2,\ldots$:

Alice announces two finite partial functions $A_{s},R_{s}\colon\mathbb{N}^{2}\rightharpoonup\mathbb{N}$.

Bob announces $s+1$ finite partial functions $B_{s}^{k}\colon\mathbb{N}^{2}\rightharpoonup\mathbb{N}\ \left(k\leq s\right)$.
\item [{Collateral duties}] $A_{s}\subseteq A_{s+1}$, $R_{s}\subseteq R_{s+1}$
and $B_{s}^{k}\subseteq B_{s+1}^{k}\ \left(k\leq s\right)$ for all
$s\in\mathbb{N}$.
\item [{Winner}] Let $A=\bigcup_{s\in\mathbb{N}}A_{s}$, $R=\bigcup_{s\in\mathbb{N}}R_{s}$
and $B^{k}=\bigcup_{k\leq s}B_{s}^{k}\ \left(k\in\mathbb{N}\right)$.
Bob wins if

\begin{enumerate}
\item for any $k\in\mathbb{N}$ and for each $i\in\mathbb{N}$, there exists
a $j\in\mathbb{N}$ such that $A\left(i,\cdot\right)=B^{k}\left(j,\cdot\right)$;
\item for any $i,j,k\in\mathbb{N}$, if $i\neq j$, then $B^{k}\left(i,\cdot\right)\neq B^{k}\left(j,\cdot\right)$;
\item for any $k\in\mathbb{N}$ and for each $i\in\mathbb{N}$, if $R\left(i,\cdot\right)$
is total, then there exists a $j\in\mathbb{N}$ such that $\bigoplus_{l\neq k}B^{l}\left(R\left(i,j\right),\cdot\right)\neq B^{k}\left(j,\cdot\right)$.
\end{enumerate}
\end{description}
Let us explain how Bob wins this game by an effective way. To guarantee
that the first two winning conditions are satisfied, Bob uses countable
copies of the winning strategy of $\mathcal{G}_{0}$ described in
the second proof of \prettyref{thm:Pou64-Theorem2}. At each move,
the strategies work one by one. The $i$-th strategy starts running
at move $i$. To guarantee the third winning condition, Bob promulgates
the following instruction for the $\braket{i,k}$-th additional assistant:\textit{
if you have reserved the $j$-th row in $B^{k}$-table, and the $\left(i,j\right)$-element
of $R$-table has been filled, and, in addition, you have reserved
the $R\left(i,j\right)$-th row in $\bigoplus_{l\neq k}B^{l}$-table,
odd-ify the $R\left(i,j\right)$-th row in $\bigoplus_{l\neq k}B^{l}$-table.}
Each instruction is done at most once. The first two winning conditions
are still satisfied. Finally, we prove that the third winning condition
is also satisfied. Suppose that $R\left(i,\cdot\right)$ is total.
The $\braket{i,k}$-th additional assistant has a permanently reserved
row $j$ in $B^{k}$-table. So, by the additional instruction, the
$\braket{i,k}$-th additional assistant does not permanently reserve
the $R\left(i,j\right)$-th row in $\bigoplus_{l\neq k}B^{l}$-table,
so $\bigoplus_{l\neq k}B^{l}\left(R\left(i,j\right),\cdot\right)\neq B^{k}\left(j,\cdot\right)$.
Thus Bob wins this game.

\end{proof}
\begin{cor}
If $\mathcal{A}$, $\mathcal{B}$ and $\mathcal{C}$ are subsets of
$\mathcal{P}^{\left(1\right)}$ satisfying the hypotheses of \prettyref{cor:another-Extension-Lemma},
then there exists a uniformly c.e. independent sequence of Friedberg
numberings of $\mathcal{A}\cup\mathcal{B}\cup\mathcal{C}$.
\end{cor}

\begin{cor}
If $\mathcal{A}$ and $\mathcal{B}$ are subsets of $\mathcal{P}^{\left(1\right)}$
satisfying the hypotheses of \prettyref{cor:Kum90b-Lemma}, then there
exists a uniformly c.e. independent sequence of Friedberg numberings
of $\mathcal{A}\cup\mathcal{B}$.
\end{cor}

\begin{cor}
If $\mathcal{A}$ is a subset of $\mathcal{P}^{\left(1\right)}$ satisfying
the hypotheses of \prettyref{cor:Kum90b-Theorem1}, then there exists
a uniformly c.e. independent sequence of Friedberg numberings of $\mathcal{A}$.
\end{cor}

\section{Examples}
\begin{example}
The class of all left-c.e. reals has a uniformly c.e. independent
sequence of Friedberg numberings.
\end{example}

\begin{proof}
In this context, by `real' we mean an infinite binary sequence, and
we identify a left-c.e. real with the c.e. set that consists of all
initial segments. It suffices to check that the class of all left-c.e.
reals satisfies the hypotheses of \prettyref{cor:Kum90b-Theorem1}.
By \cite[Theorem 4]{BKH09} the class of all left-c.e. reals has a
Friedberg numbering. Every initial segment $\sigma$ of a (left-c.e.)
real can be extended to infinitely many left-c.e. reals $\sigma0^{\omega},\sigma10^{\omega},\sigma110^{\omega}$
and so on.
\end{proof}

\begin{example}
The class of all Martin-L\"of random left-c.e. reals has a uniformly
c.e. independent sequence of Friedberg numberings.
\end{example}

\begin{proof}
By \cite[Theorem 5]{BKH09} the class of all ML-random left-c.e. reals
has a Friedberg numbering. Every initial segment $\sigma$ of a (ML-random
left-c.e.) real can be extended to infinitely many ML-random left-c.e.
reals $\sigma\Omega,\sigma0\Omega,\sigma00\Omega,\ldots$, where $\Omega$
is the halting probability.
\end{proof}

\section*{Acknowledgement}

The author is grateful to Prof. Alexander Shen for helpful discussions
and comments. He recommended the author to consider the symmetric
game $\mathcal{G}_{3}$ instead of the asymmetric game $\mathcal{G}_{2}$,
and outlined the winning strategy of $\mathcal{G}_{3}$.

\bibliographystyle{amsplain}
\nocite{*}
\bibliography{\string"Game arguments in some existence theorems of Friedberg numberings\string"}

\end{document}